\documentclass[11pt]{article}

\newcommand{\reals}{\mbox{$\mathbb R$}}
\newcommand{\comps}{\mbox{$\mathbb C$}}
\newcommand{\nats}{\mbox{$\mathbb N$}}

\newcommand{\comment}[1]{}

\usepackage{amsmath}  
\usepackage{amssymb}
\usepackage{graphicx}                   

\def\squarebox#1{\hbox to #1{\hfill\vbox to #1{\vfill}}}
\def\qed{\hspace*{\fill}
        \vbox{\hrule\hbox{\vrule\squarebox{.667em}\vrule}\hrule}\smallskip}
\newenvironment{proof}{\begin{trivlist}
  \item[\hspace{\labelsep}{\em\noindent Proof.~}]
  }{\qed\end{trivlist}}

\newtheorem{lemma}{Lemma}[section]
\newtheorem{theorem}[lemma]{Theorem}
\newtheorem{corollary}[lemma]{Corollary}

\newtheorem{definition}[lemma]{Definition}

\newtheorem{conjecture}[lemma]{Conjecture}

\def\squareforqed{\hbox{\rlap{$\sqcap$}$\sqcup$}}
\def\qed{\ifmmode\squareforqed\else{\unskip\nobreak\hfil
\penalty50\hskip1em\null\nobreak\hfil\squareforqed
\parfillskip=0pt\finalhyphendemerits=0\endgraf}\fi}

\newlength{\tablength}
\newlength{\spacelength}
\newcommand{\tabstar}{\hspace*{\tablength}}
\newcommand{\spacestar}{\hspace*{\spacelength}}
\def\obeytabs{\catcode`\^^I=\active}
{\obeytabs\global\let^^I=\tabstar}
{\obeyspaces\global\let =\spacestar}
\newenvironment{display}{\begingroup\obeylines\obeyspaces\obeytabs}{\endgroup}
\newenvironment{prog}{\begin{display}\parskip0pt\sf}{\end{display}}

\title{On the maximum number of edges of non-flowerable coin graphs}

\author{
{\sl Geir Agnarsson}
\thanks{Department of Mathematical Sciences,
George Mason University,
MS 3F2,
4400 University Drive,
Fairfax, VA -- 22030, USA,
{\tt geir@math.gmu.edu}}
\and
{\sl Jill Bigley Dunham}
\thanks{Department of Mathematics,
Hood College, 401 Rosemont Avenue, Frederick, MD -- 21701, USA,
{\tt dunham@hood.edu}} 
}

\date{}

\begin{document}

\maketitle

\begin{abstract}
For $n\in\nats$ and $3\leq k\leq n$ we compute the exact value
of $E_k(n)$, the maximum number of edges of a simple planar
graph on $n$ vertices where each vertex bounds an $\ell$-gon
where $\ell\geq k$. The lower bound of $E_k(n)$ is obtained
by explicit construction, and the matching upper bound is
obtained by using Integer Programming (IP.) We then use
this result to conjecture the maximum number of edges
of a non-flowerable coin graph on $n$ vertices. 
A {\em flower} is a coin graph representation
of the wheel graph. A collection of coins or discs in 
the Euclidean plane is {\em non-flowerable} if no
flower can be formed by coins from the collection. 

\vspace{3 mm}


\vspace{2 mm}

\end{abstract}

\section{Introduction}
\label{sec:intro}

In this article, we will prove a result which gives the maximum number of 
edges in a plane graph on $n$ vertices, where each vertex bounds some 
$\ell$-gon for $\ell \geq k$.  We will find the exact 
upper bound using integer programming and the matching lower bound 
by construction. This problem arose from an investigation of extremal 
coin graphs with multiple radii that satisfy certain conditions 
we will discuss in Section~\ref{sec:non-flowerable}.
Here a {\em coin graph} is a graph whose vertices can be represented as closed, 
non-overlapping disks in the Euclidean plane such that two vertices are adjacent
if and only if their corresponding disks intersect at their boundaries, 
i.e.~they touch. 
 
Coin graphs are ubiquitous in the discrete geometry literature
especially since the most general ones (with no restrictions on the radii
of the coins) are, by a well-known theorem of Thurston~\cite{Thurston97},
precisely the planar graphs\footnote{This theorem can also be 
attributed to Koebe~\cite{Koebe} 
and Andreev~\cite{Andreev}. Koebe's original proof covered 
only the case of fully-triangulated planar graphs. Thurston 
reduced the proof to the previous theorem of Andreev.  
Thurston's proof is of the more general case of all planar graphs.}.
One of the best known extremal problems of coin graphs is perhaps
one posed by Erd\H{o}s~\cite{E46} in 1946 and again by 
Reutter~\cite{Reutter} in 1972: for a given natural number $n$,
determine the maximum number of edges a coin graph can have
if all the coins have the same radius (called {\em unit coin graphs}.)
This problem has an unusually nice solution, due to 
Harborth~\cite{Harborth} from 1974, who showed that the
maximum number of edges is given by 
$T(n) := \left\lfloor 3n- \sqrt{12n-3} \right\rfloor$.  
This problem can be generalized in many ways, as suggested in a recent
excellent survey of open research problems in discrete 
geometry~\cite[p.~222]{Brass05}. For instance it can be generalized
to (i) graphs embedded in other surfaces such as the sphere or
to graphs embedded in $n$-dimensional Euclidean space for $n \geq 3$,
where the definition of a coin graph is modified appropriately to an 
$n$-dimensional sphere graph, or  
(ii) by adding the constraint that no three vertices of the graph can be
collinear forces the maximum degree of any vertex to be 5, leading to
a different upper bound, or  (iii) by defining a similar class of 
graphs by connecting two vertices if and only if their distance $d$ satisfies 
$1 \leq d \leq 1+\epsilon$ for some given small $\epsilon >0$.  
This structure can be pictured as a unit coin graph using elastic disks that 
can stretch some small amount. It is conjectured that for small 
$\epsilon$ (less than 0.15 times the defined unit distance) the 
maximum number of edges is still $T(n)$ as in the case of the
unit coin graph~\cite{Brass05}. Finally, (iv) Swanepoel~\cite{Swanepoel}
recently conjectured that the largest number of edges in a coin graph 
with no triangular faces is given by 
$\left\lfloor 2n-2 \sqrt{n}\right\rfloor$.  All these slight modifications are
still open problems. 
Another natural generalization  of the unit coin graph problem above,
that is not discussed in~\cite{Brass05}, is to allow coins of
more than one possible radius. Brightwell and Scheinerman~\cite{Scheinerman} 
explored integral representations of coin graphs, where the radii of the 
coins can take arbitrary positive integer values. 

The results of this article were inspired by an extremal problem
in the same vein as mentioned above, namely to determine
the exact maximum number of edges in a coin graph on $n$ coins
where their radii is such that no wheel graph can by formed by them.
Here a {\em wheel graph} is formed 
from a simple cycle by connecting one additional central
vertex to each of the vertices of the cycle. A physical 
interpretation of this is to have a collection of $n$ 
coins on the table and their sizes are such that it is
impossible to completely ``surround'' one coin with other
coins such that they all touch. This means that the underlying
plane graph of the coin graph has no vertex that only borders 
triangular faces. The article is organized as follows: in 
Section~\ref{sec:tight},
we first introduce our notation and terminology. Then for given
$n$ and $k$ we compute the tight upper bound of the maximum number of edges
a simple plane graph on $n$ vertices can have, if every vertex
borders some $\ell$-gon where $\ell\geq k$. The lower bound is obtained by
direct construction, whereas the matching upper bound is
obtained with Integer Programming (IP). Unlike many IP problems,
the one we obtain is simple enough to be able to solve completely.
In Section~\ref{sec:non-flowerable}
we give an upper bound for the maximum
number of edges a coin graph on $n$ vertices 
with no induced wheel graphs, and conjecture that this bound is
indeed tight.

\section{The general tight upper bound}
\label{sec:tight} 

\paragraph{Notation and terminology}
The set $\{1,2,3,\ldots\}$ of natural numbers
will be denoted by $\nats$. The set of real numbers is ${\reals}$
and the Euclidean plane ${\reals}^2$, the Cartesian product of two 
copies of ${\reals}$. The complex number field is ${\comps}$.
Unless otherwise stated, all 
graphs in this article will be finite, simple and undirected. 
The cycle graph on $n$ vertices and $n$ edges will be denoted
by $C_n$ and the wheel graph on $k+1$ vertices will be denoted by
$W_k$. 

Our main objective in this section is to prove the 
following theorem.
\begin{theorem}
\label{thm:kgon}
Let $k,n\in\nats$ with $n\geq k \geq 4$.  
The maximum number $E_k(n)$ of edges of a plane graph on
$n$ vertices, where each vertex bounds some $\ell$-gon for
$\ell\geq k$, is given by
\[
E_k(n) = T_k(n) := \left\lfloor \frac{(2k+3)n}{k} -6\right\rfloor- \alpha
\]
where
\begin{eqnarray*}
\alpha 	& = & \left\{ \begin{array}{lll}
  0     & \mbox{if $n \equiv k-1 \pmod k$} \\
\left\lfloor 2-\frac{6}{k}  \right\rfloor & \mbox{if $n \equiv k-2 \pmod k$} \\
\left\lfloor \frac{3 \beta}{k} \right\rfloor & 
          \mbox{if $n \equiv \beta \pmod k$ for $0 \leq \beta \leq k-3$.}
                      \end{array}
              \right.
\end{eqnarray*}
\end{theorem}
We will show that $T_k(n)$ is both an upper bound and a lower bound 
for $E_k(n)$.
We start with the easier case and show that $T_k(n)$ is lower bound using an 
explicit construction, and we consider each of the three cases, 
$n \equiv k-1, k-2, \beta \pmod k$ where $0 \leq \beta \leq k-3$, 
separately, since each case has a 
unique construction. We then conclude the section with the more involved case
and prove that the matching lower bound 
$T_k(n)$ is also an upper bound.  

\paragraph{The lower bound} 
Write $n=kj +\beta$ where $0\leq \beta \leq k-1$. 
Form $j-1$ disjoint copies of $C_k$ and one copy of $C_{k+\beta}$
in the plane, no cycle containing another cycle, consisting of $n$ edges 
altogether. We need $3(j-1)$ edges to connect the cycles into one connected 
component such that (i) the infinite face is bounded by a simple $n$-cycle and 
(ii) the internal faces of this $n$-cycle other than the $C_k$s and the 
$C_{k+\beta}$ are triangular. Then we add $n-3$ edges to fully triangulate 
the infinite face. The total number of edges thus obtained
is $e(n,j):= n +3(j-1) + (n-3)$. Consider now the various
cases for $\beta$.

{\sc First case: $\beta = k-1$.}
In this case $C_{k+\beta} = C_{2k-1}$ and two additional edges can be added 
to the interior of the cycle $C_{2k-1}$ to create 3 regions, 
2 bounded by $k$-gons and one by a triangle such that every vertex is 
bounded by a $k$-gon.  Add these additional edges 
between appropriate vertices of the cycle $C_{2k-1}$.  
The total number of edges is then given by
\[
e(n,j) + 2 
= 2n + 3j - 4 
= \frac{(2k+3)n}{k} -6 - \left( 1- \frac{3}{k}\right) 
= \left\lfloor \frac{(2k+3)n}{k} -6 \right\rfloor.
\]

{\sc Second case: $\beta = k-2$.}
Here $C_{k+\beta} = C_{2k-2}$ and one additional edge 
can be added to the interior of the cycle $C_{2k-2}$ to create 2 regions 
bounded by $k$-gons. Add this additional edge between appropriate vertices 
of the cycle $C_{2k-2}$. The total number of edges is now given by
\[
e(n,j) + 1 
= 2n + 3j - 5 
= 2n + 3\left(\frac{n+2}{k}\right)-6 -2 
= \frac{(2k+3)n}{k} -6 - \left(2 - \frac{6}{k} \right).
\]
Since for any real numbers $x, y$ with $x-y$ a positive integer we have 
$x-y = \left\lfloor x\right\rfloor- \left\lfloor y\right\rfloor$, 
then this last expression equals 
$\left\lfloor \frac{(2k+3)n}{k} -6\right\rfloor  
-\left\lfloor 2- \frac{6}{k} \right\rfloor$.

{\sc Third case: $0\leq \beta\leq k-3$.}
Here the total number of edges is given by
\[
e(n,j) 
= 2n + 3j -6 
= \frac{(2k+3)n}{k} -6 - \frac{3\beta}{k}
= \left\lfloor \frac{(2k+3)n}{k} -6\right\rfloor - 
\left\lfloor \frac{3\beta}{k}\right\rfloor,
\]
the last step just as in the previous case when $\beta = k-2$.
These three cases show that the mentioned bound $T_k(n)$ can always be reached.

\paragraph{The upper bound}
We will derive the matching upper bound using Integer Programming.
Unlike most integer programs, it turns out that our specific one in this 
case will be simple enough to be able to spot a general pattern to solve
it exactly.

Assume we have a plane graph $G$ on $n$ vertices with the property 
mentioned in the theorem.  The number of edges is $m$ and the number 
of faces is $f$.  Form a new graph $G'$ by adding a vertex inside each 
$\ell$-gon, where $\ell\geq k$ and connect that vertex to all the 
vertices bounding the $\ell$-gon.  Let $n'$, $m'$, and $f'$ be the number 
of vertices, edges, and faces of $G'$.  Note that $G'$ is planar and 
fully triangulated.  For $i \in \{3, \ldots, k-1\}$, let $f_i$ denote 
the number of $i$-sided faces of $G$ and $f_k$ be the number of all 
$\ell$-sided faces where $\ell \geq k$. Then 
$f=f_3 + f_4 + \cdots + f_{k-1} + f_k$.  By assumption we have 
$n'=n+f_4 + \cdots + f_{k-1} + f_k$ and $m'=3n'-6$, by Euler's formula.

Let $d$ be the sum of the degrees of all the vertices that were added 
above, so $d$ also equals the number of edges added to $G$ to obtain $G'$.  
Hence $m'=m+d = 3(n+f_4 + \cdots + f_k)-6$, so 
$m=3n-6-(d-3(f_4 + \cdots +f_{k-1}+f_k))$.  Note that 
$d= d_4 + d_5 + \cdots + d_{k-1} + d_k$ where for each 
$i \in \{4, \ldots, k-1\}$, $d_i$ is the sum of the degrees of the 
vertices of degree $i$ added to $G$ and $d_k$ is the sum of degrees of 
vertices of degree greater than or equal to $k$ added to $G$.  Therefore we 
have $d_i = if_i$ for each $i \in \{4, \ldots, k-1\}$ and so 
$d=4f_4 + \cdots + (k-1)f_{k-1} + d_k$ and hence 
\begin{equation}
\label{eqn:m}
m = 3n-6-(f_4 + 2f_5 + \cdots + (k-3)f_{k-1} + d_k -3f_k).
\end{equation}
Note that $m$ is maximized if $f_4 + 2f_5 + \cdots + (k-3)f_{k-1} + d_k -3f_k$ 
is minimized. Since the conditions are (1) $n \leq d_k$, (2) $f_i \geq 0$ for 
$i \in \{4, \ldots, k\}$, and (3) $kf_k \leq d_k$, we can simplify this 
optimization problem by setting $f_i=0$ for $i = 4,\ldots, k-1$ and the 
problem reduces to minimizing the value of $d_k-3f_k$ over 
nonnegative integers, 
given the constraints $d_k \geq n$ and $k f_k \leq d_k$.
\begin{lemma}
\label{lmm:key}
If $k,n\in\nats$ and $n\geq k \geq 4$ and 
$\mu(n,k):= \min \{x-3y : 
x,y \in \mathbb{N} \cup \{0\}, x \geq n, ky \leq x \}$, 
then
\[
\mu(n,k) = n + \gamma- 3\left\lfloor\frac{n+\gamma}{k}\right\rfloor
\mbox{ where } 
\gamma 	= \left\{ \begin{array}{lll}
                     1    & \mbox{if $n \equiv k-1 \pmod k$} \\
                     2    & \mbox{if $n \equiv k-2 \pmod k$} \\
                     0    & \mbox{otherwise.}
                    \end{array}
          \right.
\]
\end{lemma}
\begin{proof}
Drawing the vector $(1,-3)$ and the lines $x=n$ and $ky = x$ in the Euclidean 
plane ${\reals}^2$,
we can spot the solution to our Integer Program $\mu(n,k)$, since 
the function $x-3y = (1,-3)\cdot (x,y)$, a dot product of two vectors, 
will obtain its minimum value at $x=n$ 
and $y=\left\lfloor\frac{n}{k}\right\rfloor$ in the case of 
$n \equiv i \pmod k$ 
where $i = 0, 1, \ldots, k-3$, and at $x= k\left\lceil \frac{n}{k}\right\rceil$ 
and $y=\left\lceil \frac{n}{k}\right\rceil
=\left\lfloor\frac{x}{k}\right\rfloor$ 
otherwise. The Figures~\ref{fig:n8k6} and~\ref{fig:n7k4} illuminate 
this general 
pattern, which here remains the same for all other values of $n$ and $k$. 
Using the above definition of $\gamma$ in the lemma, we can 
write $x=n + \gamma$ as the $x$-value that will always minimize the function.  
Then we have $y=\left\lfloor\frac{x+\gamma}{k}\right\rfloor$ as the $y$-value 
that will always minimize the function.
\end{proof}
\begin{figure}
\centering
\includegraphics[scale=1]{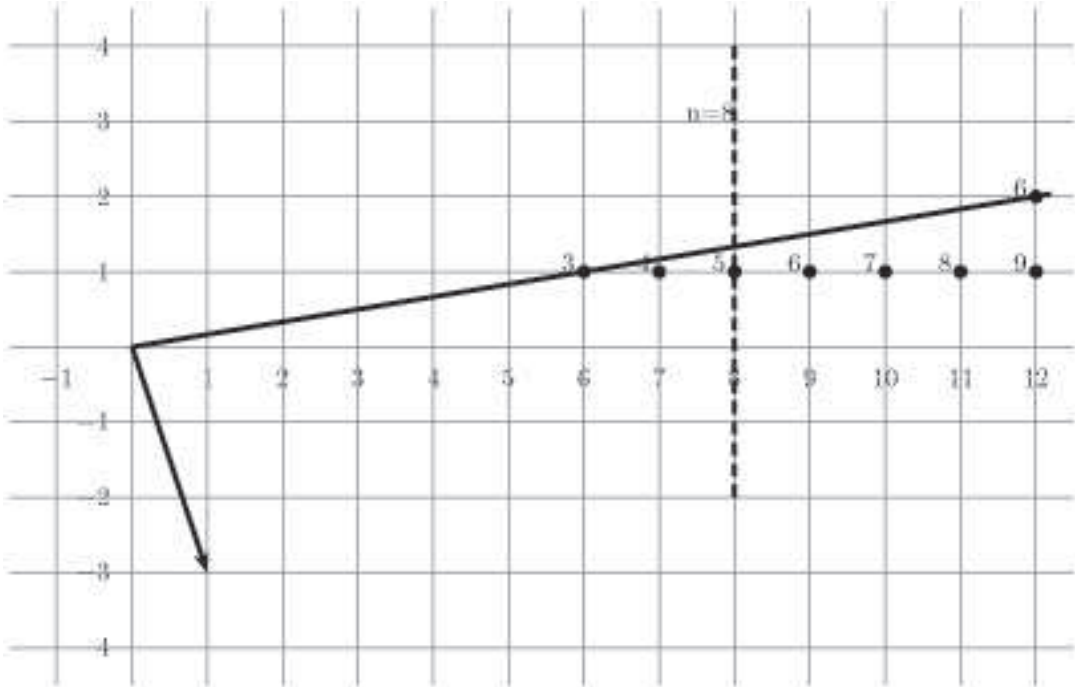} 
\caption[When $k=6$ and $n=8$, the function is minimized at 
$x=8$, $y=1$.]{When $k=6$ and $n=8$, the function is minimized at $x=8$, $y=1$.}
\label{fig:n8k6}
\end{figure}
\begin{figure}
\centering
\includegraphics{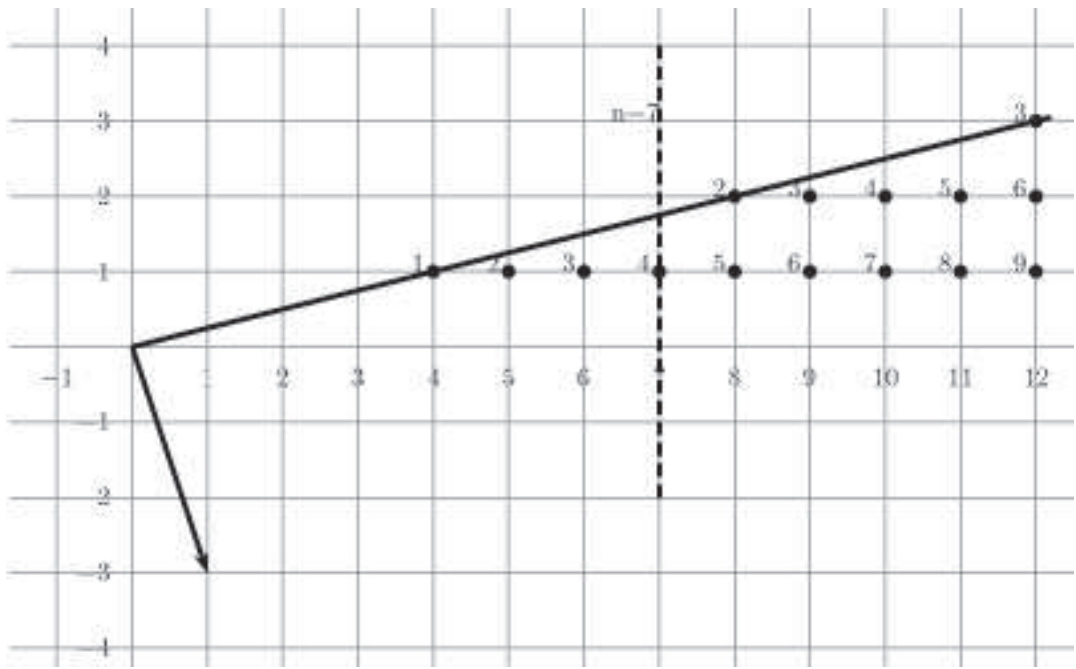} 
\caption[When $k=4$ and $n=7$, the function is minimized at 
$x=8$, $y=2$.]{When $k=4$ and $n=7$, the function is minimized at $x=8$, $y=2$.}
\label{fig:n7k4}
\end{figure}
Continuing to obtain the upper bound of $m$ from (\ref{eqn:m}), we
have by Lemma~\ref{lmm:key} that $d_k-3f_k$ is minimized when 
$d_k=n + \gamma$ and $f_k = \left\lfloor\frac{n+\gamma}{k}\right\rfloor$
and hence we have 
\[
m 
= 3n-6-(d_k-3f_k)
\leq 3n-6- n-\gamma + 3 \left\lfloor \frac{n+\gamma}{k}\right\rfloor
= 2n - 6 + 3 \left\lfloor \frac{n+\gamma}{k}\right\rfloor - \gamma.
\]

If $n \equiv k-1 \pmod k$, then $\gamma = 1$ and
\[
m
\leq 2n -6+3\left\lfloor \frac{n+1}{k}\right\rfloor-1
= \frac{(2k+3)n}{k} - 6- \left(1-\frac{3}{k}\right)
=  \left\lfloor \frac{(2k+3)n}{k} - 6\right\rfloor.
\]

If $n \equiv k-2 \pmod k$ then $\gamma =2$ and
\[
m
\leq 2n - 6 + 3 \left\lfloor \frac{n+2}{k}\right\rfloor - 2
= \frac{(2k+3)n}{k} - 6+ \left(2-\frac{6}{k}\right)
= \left\lfloor \frac{(2k+3)n}{k} - 6\right\rfloor + 
\left\lfloor 2-\frac{6}{k}\right\rfloor.
\]

If $n \equiv \beta$ where $\beta \in \{0, 1, \ldots, k-3\}$ then 
$\gamma = 0$ and 
\[
m
\leq 2n -6+3\left\lfloor \frac{n}{k}\right\rfloor
= 2n-6+\left( \frac{n-\beta}{k}\right)
=  \frac{(2k+3)n}{k} - 6 - \left(\frac{3\beta}{k}\right)
=  \left\lfloor \frac{(2k+3)n}{k} - 6 \right\rfloor- 
              \left\lfloor \frac{3\beta}{k}\right\rfloor.
\]
The above three cases show that $m\leq T_k(n)$, the matching lower
bound. This proves Theorem~\ref{thm:kgon} that $E_k(n) = T_k(n)$.

In the especially interesting case when $k=4$, the discrepancy term 
$\alpha \in \{0, \left\lfloor 2- \frac{6}{k}\right\rfloor, 
\left\lfloor \frac{3\beta}{k}\right\rfloor\}$ for 
$0 \leq \beta \leq k-3$, will be 0 in all cases, 
and hence we obtain the following:
\begin{corollary}
\label{cor:E4}
The maximum number of edges $E_4(n)$ of a plane graph on
$n$ vertices, where each vertex bounds some $\ell$-gon for
$\ell\geq 4$, is given by
\[
E_4(n) =  \left\lfloor \frac{11}{4}n - 6\right\rfloor.
\]
\end{corollary}

\section{Non-flowerable coins}
\label{sec:non-flowerable}

A coin graph representation of the wheel graph is called 
a {\em flower}. A coin graph with no flowers is {\em non-flowered}
and a collection $\mathcal{C}$ of coins is {\em non-flowerable} if no
flower can be formed by coins from $\mathcal{C}$. This terminology
is consistent with that found in~\cite{Ste}. Note that it is not
necessary for a non-flowerable collection to contain coins
of distinct radii, but it cannot contain seven or more coins of 
equal radii, since seven coins with the same radius can form a 
regular hexagonal flower. 
\begin{definition}
\label{def:collection}
For $n\in\nats$ denote by $\widetilde{\mathcal{NF}}(n)$ the set of all 
non-flowerable collections of $n$ coins. For each 
$\mathcal{C}\in\widetilde{\mathcal{NF}}(n)$ let $NF(\mathcal{C})$ 
denote the maximum number of edges
of a coin graph formed from coins in $\mathcal{C}$. Finally let
\[
NF(n) = \max(\{ NF(\mathcal{C}) : \mathcal{C}\in\widetilde{\mathcal{NF}}(n)\}).
\]
\end{definition}
Note that every non-flowered coin graph must have each coin 
bounded by an $\ell$-gon where $\ell\geq 4$. Hence, by
Corollary~\ref{cor:E4} we obtain the following corollary.
\begin{corollary}
\label{cor:NF}
For $n\in\nats$ we have
\[
NF(n) \leq E_4(n) =  \left\lfloor \frac{11}{4}n - 6\right\rfloor.
\]  
\end{corollary} 
Whether $NF(n) = E_4(n)$ or not is unknown to us as of writing
this article. 
\begin{conjecture}
\label{conj:NF}
For $n\in\nats$ we have 
$NF(n) = E_4(n) =  \left\lfloor \frac{11}{4}n - 6\right\rfloor$.
\end{conjecture}
{\sc Remark:}
Given $n\in\nats$. By Thurston's theorem~\cite{Thurston97}
we can obtain a coin graph representation of each of the planar
graphs constructed for the lower bound 
of Theorem~\ref{thm:kgon} in Section~\ref{sec:tight}.
By construction it is guaranteed that it will be non-flowered.
However, we do not know if the the underlying {\em collection} of coins  
used in this representation is non-flowerable,
since some flower could be formed by a subset of them.  We do
suspect that each such coin graph representation of the graphs formed
for the lower bound in Theorem~\ref{thm:kgon} can be represented
by a non-flowerable collection of coins: Recall that the map
$\comps \rightarrow \comps$ given by $z\mapsto 1/z$ is 
an {\em inversion} about the unit circle centered at origin.
It is known fact in plane geometry that every inversion of the
complex plane maps a coin graph to another coin graph with
the same underlying planar graph. However, the radii of coins
have all changed. We suspect that a proof of 
Conjecture \ref{conj:NF} can be obtained by inverting
a carefully chosen embedding of a coin graph on $n$ coin 
with the maximum number $E_4(n)$ of edges, resulting in a representation
using non-flowerable collection of coins. However, as far as our investigation
goes, we will stop here for the moment.

\subsection*{Acknowledgments}  

Thanks to Konrad J.~Swanepoel for related references.

\flushright{\today}

\end{document}